\documentclass[12pt]{article}

\usepackage{amssymb,a4}
\usepackage{amsmath,amsfonts,amssymb,amsthm, mathrsfs}

\newcommand{\1}{\mathbf {1}}
\newcommand{\Z}{{\mathbb Z}}
\newcommand{\C}{{\mathbb C}}
\newcommand{\g}{{\mathfrak g}}

\newcommand{\h}{{\mathfrak h}}
\newcommand{\wh}{{\widehat{\mathfrak h}}}

\newcommand{\mraff}{\mathrm{aff}}
\def\<{\langle}
\def\>{\rangle}
\def\a{\alpha}

\newcommand{\la}{\langle}
\newcommand{\ra}{\rangle}

\newtheorem{thm}{Theorem}[section]
\newtheorem{prop}[thm]{Proposition}
\newtheorem{lem}[thm]{Lemma}

\newtheorem{rmk}[thm]{Remark}
\newtheorem{definition}[thm]{Definition}

\begin{document}

\begin{center}
{\Large \bf  Representations of the orbifold of parafermion vertex operator algebra $K(osp(1|2),k)$}

\end{center}

\begin{center}
{ Cuipo Jiang$^{a}$\footnote{Supported by China NSF grants No.12171312.}
and Qing Wang$^{b}$\footnote{Supported by
China NSF grants No.12071385 and No.12161141001.}\\
$\mbox{}^{a}$ School of Mathematical Sciences, Shanghai Jiao Tong University, Shanghai 200240, China\\
\vspace{.1cm}
$\mbox{}^{b}$ School of Mathematical Sciences, Xiamen University,
Xiamen 361005, China\\
}
\end{center}

\begin{abstract}
This paper is about the orbifold theory of parafermion vertex operator algebras $K(osp(1|2),k)$ associated to the affine vertex operator superalgebra $L_{\widehat{osp(1|2)}}(k,0)$ with any positive integer $k$. Among the main results, we classify the irreducible modules for the orbifold of the parafermion vertex operator algebra $K(osp(1|2),k)$.

\end{abstract}

\section{Introduction}
\def\theequation{1.\arabic{equation}}
\setcounter{equation}{0}

This paper is a continuation in a series of papers on the study of the orbifold theory of coset vertex operator algebra $K(\mathfrak{g},k)$ called parafermion vertex operator algebra with $\g$ a finite dimensional simple Lie superalgebra. It was proved that $L_{\hat{\g}}(k,0)$ is $C_2$-cofinite and rational if and only if $\g$ is either a simple Lie algebra or $\g=osp(1|2n)$, and $k$ is a positive integer \cite{FZ}, \cite{DL}, \cite{Li1}, \cite{AL}, \cite{CL}. The commutant $K(osp(1|2n),k)$ of a Heisenberg vertex operator subalgebra in the simple affine vertex operator superalgebra $L_{\hat{\mathfrak{g}}}(k,0)$ with $k$ being a positive integer attract more attention because they are interesting examples of rational vertex operator superalgebras and they may relate to certain important $W$-superalgebras. The structure of the parafermion vertex operator algebra $K(osp(1|2n),k)$ has been studied in \cite{JW3}, we determine the generators of $K(osp(1|2n),k)$, it shows that $K(sl_2,k)$ and $K(osp(1|2),k)$ are building block of  $K(osp(1|2n),k)$ as the role of $sl_2$ played in Kac-Moody Lie algebras. So it is important to understand the representation theories of $K(sl_2,k)$ and $K(osp(1|2),k)$ first for the further study of the representation theory of the rational vertex operator algebra $K(osp(1|2n),k)$. The representation theories for $K(sl_2,k)$ and the orbifold of $K(sl_2,k)$ with $k$ being a positive integer were studied in \cite{DLY2}, \cite {DLWY}, \cite{DW1}, \cite{DW2}, \cite{DW3}, \cite{Lam}, \cite{ALY1}, \cite{ALY2}, \cite{JW1}, \cite{JW2} etc. Note that $K(osp(1|2),k)$ with $k$ being a positive integer is an extension of $C_2$-cofinite and rational vertex operator algebra $K(sl_2,k)\otimes L^{Vir}(c_{2k+3,k+2},0)$, where $L^{Vir}(c_{2k+3,k+2},0)$ is the minimal Virasoro vertex operator algebra with central charge $c_{2k+3,k+2}=1-\frac{6(k+1)^{2}}{(2k+3)(k+2)}$. From \cite{HKL}, we see that $K(osp(1|2),k)$ with positive integers $k$ are $C_2$-cofinite and rational. Also the representation theory for $K(osp(1|2),k)$ has been studied in \cite{CFK}. Next it is natural to consider the orbifold theory of the parafermion vertex operator algebra $K(osp(1|2),k)$ with $k$ being a positive integer. We proved in \cite{JW4} that the automorphism group of the parafermion vertex operator algebra $K(osp(1|2),k)$ with $k\geq 3$ is of order 2 and generated by $\sigma$ which is determined by $\sigma(h)=-h, \ \sigma(e)=f,\ \sigma(f)=e,\ \sigma(x+y)=-\sqrt{-1}(x+y),\ \sigma(x-y)=\sqrt{-1}(x-y)$, where $\{e,f,h,x,y\}$ is the basis of the Lie superalgebra $osp(1|2)$ with
$[e,f]=h,$ $[h,e]=2e$, $[h,f]=-2f$
$[h,x]=x$, $[f,x]=-y$,
$[h,y]=-y$,  $[e,y]=-x$,
$\{x,x\}=2e,$ $\{x,y\}=h$, $\{y,y\}=-2f.$ We denote the simple affine vertex operator superalgebra $L_{\widehat{osp(1|2)}}(k,0)$ by $L_{\hat{\g}}(k,0)$ in this paper. The orbifold theory of simple affine vertex operator superalgebra $L_{\hat{\g}}(k,0)$ has been studied in \cite{JW4}. The irreducible modules have been classified and the fusion rules have been determined therein.

From \cite{M2}, \cite{CM} (also see \cite{CKLR}), we know that $K(osp(1|2),k)^{\sigma}$ with positive integers $k\geq 3$ are $C_2$-cofinite and rational. In this paper, we classify the irreducible modules of $K(osp(1|2),k)^{\sigma}$. It was proved in \cite{CL} (see the different proof in \cite{Dra}) that the commutant Com$(L_{\hat{sl_2}}(k,0), L_{\hat{\g}}(k,0))=L^{Vir}(c_{2k+3,k+2},0),$ and Com$(L^{Vir}(c_{2k+3,k+2},0), L_{\hat{\g}}(k,0))=L_{\hat{sl_2}}(k,0).$ We also have
Com$(V_{\mathbb{Z}\gamma},L_{\hat{sl_2}}(k,0))=K(sl_2,k)$ and Com$(K(sl_2,k),L_{\hat{sl_2}}(k,0))=V_{\mathbb{Z}\gamma}$ \cite{DLY2}. From the decomposition of the affine vertex operator superalgebra $L_{\hat{\g}}(k,0)$ as the $L_{\hat{sl_2}}(k,0)\otimes L^{Vir}(c_{2k+3,k+2},0)$-module \cite{CFK}, we decompose $L_{\hat{sl_2}}(k,0)^{\sigma}$ as $V_{\mathbb{Z}\gamma}^{+}\otimes K(sl_2,k)^{\sigma}$-modules. Since $V_{\mathbb{Z}\gamma}^{+}\otimes K(\g,k)^{\sigma}$ are conformal embedding in $L_{\hat{\g}}(k,0)^{\sigma}$, and $V_{\mathbb{Z}\gamma}^{+}\otimes K(sl_2,k)^{\sigma}$ are conformal embedding in $L_{\hat{sl_2}}(k,0)^{\sigma}$, together with Com$(V_{\mathbb{Z}\gamma}^{+},L_{\hat{\g}}(k,0)^{\sigma})=K(osp(1|2),k)^{\sigma}$ and Com$(K(osp(1|2),k)^{\sigma},L_{\hat{\g}}(k,0)^{\sigma})=V_{\mathbb{Z}\gamma}^{+},$ we can obtain all the irreducible modules of
$K(\g,k)^{\sigma}$ from the decomposition of the irreducible modules of $L_{\hat{\g}}(k,0)^{\sigma}$ as $V_{\mathbb{Z}\gamma}^{+}\otimes K(\g,k)^{\sigma}$-modules. We point out that all the irreducible modules of $L_{\hat{\g}}(k,0)^{\sigma}$ have been classified in \cite{JW4}. The difficulty is how to distinguish the inequivalent irreducible modules of $K(\g,k)^{\sigma}$ among these irreducible modules of $K(\g,k)^{\sigma}$ we get from the decomposition. We clarify this in the proof of Theorem \ref{thm:irr para}.

The paper is organized as follows. In Section 2, we recall some results about the affine vertex operator superalgebra $L_{\hat{\g}}(k,0)$ and its parafermion vertex operator subalgebra $K(osp(1|2),k)$ and decompose all the irreducible modules of $L_{\hat{sl_2}}(k,0)^{\sigma}$ as $V_{\mathbb{Z}\gamma}^{+}\otimes K(sl_2,k)^{\sigma}$-modules. In Section 3, we classify the irreducible modules of the orbifold $K(\g,k)^{\sigma}$ and prove that these irreducible modules of the orbifold $K(\g,k)^{\sigma}$ are self-dual.

\section{Preliminaries}
\label{Sect:V(k,0)}\def\theequation{2.\arabic{equation}}

In this section, we recall from \cite{JW2}, \cite{JW3} some basic results on the affine vertex operator superalgebra and parafermion vertex
operator algebra of $osp(1|2)$ at level $k$ with $k$ being a positive integer.

Let $\g$ be the finite dimensional simple Lie superalgebra $osp(1|2)$ with a Cartan
subalgebra $\h.$ Let $\{e,f,h,x,y\}$ be the basis of the Lie superalgebra $\g$ with the anti-commutation relations:
$$[e,f]=h,\; [h,e]=2e, \; [h,f]=-2f$$
$$[h,x]=x, \; [e,x]=0, \; [f,x]=-y$$
$$[h,y]=-y, \; [e,y]=-x, \; [f,y]=0$$
$$\{x,x\}=2e, \; \{x,y\}=h, \; \{y,y\}=-2f.$$ Let  $\la ,\ra$ be an invariant even supersymmetric
nondegenerate bilinear form on $\g$ such that $\<\a,\a\>=2$ if
$\alpha$ is a long root in the root system of even, where we have identified $\h$ with $\h^*$
via $\<,\>.$  Let $\hat{\mathfrak g}= \g \otimes \C[t,t^{-1}] \oplus \C K$
be the corresponding affine Lie superalgebra. Let $k$ be a positive
integer and
\begin{equation*}
 V_{\hat{\g}}(k,0) = Ind_{\g \otimes
\C[t]\oplus \C K}^{\hat{\g}}\C
\end{equation*}
the induced $\hat{\g}$-module such that ${\g} \otimes \C[t]$ acts as $0$ and $K$ acts as $k$ on $\1=1$.

We denote by $a(n)$ the operator on $V_{\hat{\g}}(k,0)$ corresponding to the action of
$a \otimes t^n$. Then
$$[a(m), b(n)] = [a,b](m+n) + m \la a,b \ra \delta_{m+n,0}k$$
for $a, b \in \g$ and $m,n\in \Z$.

Let $a(z) = \sum_{n \in \Z} a(n)z^{-n-1}$. Then $V_{\hat{\g}}(k,0)$ is a
vertex operator superalgebra generated by $a(-1)\1$ for $a\in \g$ such
that $Y(a(-1)\1,z) = a(z)$ with the
vacuum vector $\1$ and the Virasoro vector
\begin{align*}
\omega_{\mraff} &= \frac{1}{2(k+\frac{3}{2})} \Big(
\frac{1}{2}h(-1)h(-1)\1 +e(-1)f(-1)\1+f(-1)e(-1)\1
\\&-
\frac{1}{2}x(-1)y(-1)\1+
\frac{1}{2}y(-1)x(-1)\1
\Big)
\end{align*}
of central charge $\frac{2k}{2k+3}$(see \cite{KRW}).

Let $M(k)$ be the vertex operator subalgebra of $V_{\hat{\g}}(k,0)$
generated by $h(-1)\1$ with the Virasoro
element
$$\omega_{\gamma} = \frac{1}{4k}
h(-1)^{2}\1$$
of central charge $1$.


The vertex operator superalgebra $V_{\hat{\g}}(k,0)$ has a unique maximal ideal $\mathcal{J}$, which is generated by a weight $k+1$ vector $e(-1)^{k+1}\1$ \cite{GS}. The quotient algebra $L_{\hat{\g}}(k,0)=V_{\hat{\g}}(k,0)/\mathcal{J}$ is a
simple, rational and $C_2$-cofinite vertex operator superalgebra \cite{AL}, \cite{CL}. Moreover, the image of $M(k)$
in $L(k,0)$ is isomorphic to $M(k)$ and will be
denoted by $M(k)$ again. Set
\begin{equation*}
 K(\g,k)=\{v \in L_{\hat{g}}(k,0)\,|\, h(m)v =0
\text{ for }\; h\in {\mathfrak h},
 m \ge 0\}.
\end{equation*}
Then $K(\g,k)$, which is the space of highest weight vectors
with highest weight $0$ for $\wh$,
is the commutant of $M(k)$ in $L_{\hat{\g}}(k,0)$. Note that $K(\g,k)$ can be also viewed as the commutant of the lattice vertex operator algebra $V_{\mathbb{Z}\gamma}$ with $\langle\gamma,\gamma\rangle=2k$ (see \cite{DLY2}). We proved in \cite{JW3} that $K(\g,k)$ is a simple vertex operator algebra which is called the parafermion vertex operator algebra and generated by

\begin{equation}\label{eq:w3}
\begin{split}
\omega
=\frac{1}{2k(k+2)}(-h(-1)^{2}{\1}
+2ke(-1)f(-1){\1}-kh(-2){\1}),
\end{split}
\end{equation}

\begin{equation}\label{eq:w3'}
\begin{split}
\bar{\omega}
=-h(-1)^{2}{\1}
+4kx(-1)y(-1){\1}-2kh(-2){\1},
\end{split}
\end{equation}

\begin{equation}\label{eq:W3'}
\begin{split}
W^3 &= k^2 h(-3){\1} + 3 k
h(-2)h(-1){\1} + 2h(-1)^3{\1}
-6k h(-1)e(-1)f(-1){\1}
\\
&+3k^2e(-2)f(-1){\1}
-3k^2e(-1)f(-2){\1},
\end{split}
\end{equation}

\begin{equation}\label{eq:W3'''}
\begin{split}
\bar{W}^3 &= 2k^2 h(-3){\1} + 3 k
h(-2)h(-1){\1} + h(-1)^3{\1}
-6k h(-1)x(-1)y(-1){\1}\\&
+6k^2x(-2)y(-1){\1}-6k^2x(-1)y(-2){\1}.\end{split}
\end{equation}

 Let $L_{\hat{sl_2}}(k,0)$ be the simple affine vertex operator algebra associated to $\hat{sl_2}$ and let $L(k,i)$ for $0\leq i\leq k$ be the irreducible modules for the rational vertex operator algebra $L_{\hat{sl_2}}(k,0)$ with the top level $U^{i}=\bigoplus_{j=0}^{i}\mathbb{C}v^{i,j}$ which is an $(i+1)$-dimensional irreducible module of the simple Lie algebra $\C h(0)\oplus\C e(0)\oplus \C f(0)\cong sl_2$.

Let $\sigma$ be an automorphism of the Lie superalgebra $osp(1|2)$ defined by $\sigma(h)=-h, \ \sigma(e)=f,\ \sigma(f)=e,\ \sigma(x+y)=-\sqrt{-1}(x+y),\ \sigma(x-y)=\sqrt{-1}(x-y)$. $\sigma$ can be lifted to an automorphism $\sigma$ of the vertex operator superalgebra $V_{\hat{\g}}(k,0)$ of order 4 in the following way:
$$\sigma(a_{1}(-n_{1})\cdots a_{s}(-n_{s})\1)=\sigma(a_{1})(-n_{1})\cdots \sigma(a_{s})(-n_{s})\1$$
for $a_{i}\in osp(1|2)$ and $n_{i}>0$. Then $\sigma$ induces an automorphism of $L_{\hat{\g}}(k,0)$ as $\sigma$ preserves the unique maximal ideal $\mathcal{J}$, and the Virasoro element $\omega_{\gamma}$ is invariant under $\sigma$. Thus $\sigma$ induces an automorphism of the parafermion vertex operator algebra $K(\g,k)$. In fact, $\sigma(\omega)=\omega,\ \sigma(\bar{\omega})=\bar{\omega}, \sigma(W^{3})=-W^{3},\ \sigma(\bar{W^{3}})=-\bar{W^{3}}$. We note that $\sigma(h)=-h, \ \sigma(e)=f,\ \sigma(f)=e$ induces an autmorphism of the  parafermion vertex operator algebra $K(sl_2,k)$, the irreducible modules of the orbifold vertex operator algebra $K(sl_2,k)^{\sigma}$ are classified in \cite{JW1} and fusion rules are determined in \cite{JW2}.

Let $V_{{\mathbb{Z}}\gamma}$ be the vertex operator algebra associated with a rank one lattice $\mathbb{Z}\gamma$ and $\langle\gamma,\gamma\rangle=2k$. It is known that $V_{{\mathbb{Z}}\gamma}$ is generated by $e^{\gamma}=\frac{1}{k!}e(-1)^{k}\1$ and $e^{-\gamma}=\frac{1}{k!}f(-1)^{k}\1$. $V_{{\mathbb{Z}}\gamma}$ has an order $2$ automorphism induced from $\theta(\gamma)=-\gamma$, and the $\theta$-invariants $V_{{\mathbb{Z}}\gamma}^{+}$ is a simple vertex operator subalgebra of $V_{{\mathbb{Z}}\gamma}$ and $V_{{\mathbb{Z}}\gamma}^{+}$ is generated by $e^{\gamma}+e^{-\gamma}$ \cite{DN}. The irreducible modules of $V_{{\mathbb{Z}}\gamma}^{+}$ are classified in \cite{DN}.
\begin{prop}\cite{DN}\label{pro:orbifoldlattice} All the irreducible modules of $V_{{\mathbb{Z}}\gamma}^{+}$ are:

\begin{eqnarray*}V_{\mathbb{Z}\gamma}^{+},\ V_{\mathbb{Z}\gamma}^{-}, \ V_{\mathbb{Z}\gamma+\frac{i}{2k}} \ (1\leq i \leq k-1),\ V_{\mathbb{Z}\gamma+\frac{1}{2}\gamma}^{+},\ V_{\mathbb{Z}\gamma+\frac{1}{2}\gamma}^{-},\end{eqnarray*}
\begin{eqnarray*}V_{\mathbb{Z}\gamma}^{T_{1},+},\ V_{\mathbb{Z}\gamma}^{T_{1},-}, \ V_{\mathbb{Z}\gamma}^{T_{2},+},\ V_{\mathbb{Z}\gamma}^{T_{2},-}.\end{eqnarray*}
\end{prop}

 Let $L^{Vir}(c_{p,q},0)$ be the minimal Virasoro vertex operator algebra with central charge $c_{p,q}=1-\frac{6(p-q)^{2}}{pq}$, $p,q\in\mathbb{Z}_{\geq 2}, (p,q)=1.$  It is known that \cite{W} $L^{Vir}(c_{p,q},0)$ is rational and its irreducible modules are $\{L^{Vir}(c_{p,q},h_{p,q}^{r,s})|h_{p,q}^{r,s}\in S\}$, where

$$S=\{h_{p,q}^{r,s}=\frac{(sq-rp)^{2}-(p-q)^{2}}{4pq}|1\leq r\leq q-1, 1\leq s\leq p-1\}.$$ We denote these irreducible modules $L^{Vir}(c_{p,q},h_{p,q}^{r,s})$ by  $V_{r,s}$ for $1\leq r\leq q-1, 1\leq s\leq p-1$.

 The following theorem gives the classification of the irreducible modules of the orbifold vertex operator algebra $K(sl_2,k)^{\sigma}$ for $k\geq 3$ \cite{JW1}. For simplicity, we denote $K(sl_2,k)^{\sigma}$ by $K^{+}$ in this paper.

 \begin{thm}\cite{JW1}\label{thm:orbifold3'}
If $k=2n+1$, $n\geq 1$, there are $\frac{(k+1)(k+7)}{4}$ inequivalent irreducible modules of $K^{+}$.
 If $k=2n$, $n\geq 2$, there are $\frac{(k^{2}+8k+28)}{4}$ inequivalent irreducible modules of $K^{+}$. More precisely, if $k=2n+1$, $n\geq 1$, the set
 \begin{eqnarray*}
&& \{ W(k,i)^{j} \ \mbox{for} \  0\leq i\leq \frac{k-1}{2}, j=1,2,\\
&&(M^{i,j})^{s} \ \mbox{for} \  (i,j)=(i,\frac{i}{2}), i=2,4,6,\cdots,2n, \ \mbox{and} \ (i,j)=(2n+1,0), s=0,1,\\ && M^{i,0} \ \mbox{for} \ 1\leq i\leq \frac{k-1}{2},
  M^{i,j} \ \mbox{for} \ 3\leq i\leq k, \mbox{if} \  i=2m, 1\leq j\leq m-1, \mbox{if} \ i=2m+1, 1\leq j\leq m\} \\
 \end{eqnarray*} gives all inequivalent irreducible $K^{+}$-modules. If $k=2n$, $n\geq 2$, the set
 \begin{eqnarray*}
&& \{ W(k,i)^{j} \ \mbox{for} \  0\leq i\leq \frac{k}{2}, j=1,2,  \widetilde{W(k,\frac{k}{2})}^{j} \ \mbox{for} \ j=1,2,\\
&&(M^{i,j})^{s} \ \mbox{for} \  (i,j)=(i,\frac{i}{2}), i=2,4,6,\cdots,2n, (i,j)=(n,0) \mbox{and} \ (i,j)=(2n,0), s=0,1,\\ && M^{i,0} \ \mbox{for} \ 1\leq i\leq \frac{k-2}{2},
  M^{i,j} \ \mbox{for} \ 3\leq i\leq k, \mbox{if} \  i=2m, 1\leq j\leq m-1, \mbox{if} \ i=2m+1, 1\leq j\leq m\} \\
 \end{eqnarray*} gives all inequivalent irreducible $K^{+}$-modules.
\end{thm}

\begin{rmk}\label{rmk:orbifold3'} With the notations in  Theorem \ref{thm:orbifold3'}, we call $W(k,i)^{j}$ and $\widetilde{W(k,\frac{k}{2})}^{j}$ twisted type modules and  $(M^{i,j})^{s}, M^{i,j}$ untwisted modules of type $I$ and type $II$ respectively.

To emphasize the action of the automorphism $\sigma$, we denote twisted type modules $W(k,i)^{1}$ by $W(k,i)^{+}$ and $W(k,i)^{2}$ by $W(k,i)^{-}$, and we denote $\widetilde{W(k,\frac{k}{2})}^{1}$ by $\widetilde{W(k,\frac{k}{2})}^{+}$ and  $\widetilde{W(k,\frac{k}{2})}^{2}$ by $\widetilde{W(k,\frac{k}{2})}^{-}$. We denote untwisted modules $(M^{i,j})^{0}$ of type $I$ by $(M^{i,j})^{+}$ and  $(M^{i,j})^{1}$ by $(M^{i,j})^{-}$.

\end{rmk}

We need the following lemmas for later use.

\begin{lem}\label{decomp.} Let $k$ be a positive integer and $0\leq i\leq k$, we have

(1) If $i\in 2\mathbb{Z}, k\in 2\mathbb{Z}$,
\begin{eqnarray*}
	\begin{split}L(k,i)^{+} &= (\oplus_{0\leq j\leq \frac{i}{2}-1, (i,j)\neq(k,0)}V_{\mathbb{Z}\gamma+\frac{i-2j}{2k}\gamma}\otimes M^{i,j})\oplus V_{\mathbb{Z}\gamma}^{+}\otimes (M^{i,\frac{i}{2}})^{+}\oplus V_{\mathbb{Z}\gamma}^{-}\otimes (M^{i,\frac{i}{2}})^{-}\\
		&\oplus V_{\mathbb{Z}\gamma+\frac{1}{2}\gamma}^{+}\otimes (M^{k-i,\frac{k-i}{2}})^{+}\oplus V_{\mathbb{Z}\gamma+\frac{1}{2}\gamma}^{-}\otimes (M^{k-i,\frac{k-i}{2}})^{-}
		\oplus(\oplus_{1\leq j\leq\frac{k-i}{2}-1}V_{\mathbb{Z}\gamma+\frac{-i-2j}{2k}\gamma}\otimes M^{i,i+j}),
	\end{split}
\end{eqnarray*}

\begin{eqnarray*}
	\begin{split}L(k,i)^{-} &= (\oplus_{0\leq j\leq \frac{i}{2}-1, (i,j)\neq(k,0)}V_{\mathbb{Z}\gamma+\frac{i-2j}{2k}\gamma}\otimes M^{i,j})\oplus V_{\mathbb{Z}\gamma}^{+}\otimes (M^{i,\frac{i}{2}})^{-}\oplus V_{\mathbb{Z}\gamma}^{-}\otimes (M^{i,\frac{i}{2}})^{+}\\
		&\oplus V_{\mathbb{Z}\gamma+\frac{1}{2}\gamma}^{+}\otimes (M^{k-i,\frac{k-i}{2}})^{-}\oplus V_{\mathbb{Z}\gamma+\frac{1}{2}\gamma}^{-}\otimes (M^{k-i,\frac{k-i}{2}})^{+}
		\oplus(\oplus_{1\leq j\leq\frac{k-i}{2}-1}V_{\mathbb{Z}\gamma+\frac{-i-2j}{2k}\gamma}\otimes M^{i,i+j}),
	\end{split}
\end{eqnarray*}
as $V_{\mathbb{Z}\gamma}^{+}\otimes K^{+}$-modules.

(2)If $i\in 2\mathbb{Z}+1, k\in 2\mathbb{Z}$,
\begin{eqnarray*}
\begin{split}L(k,i)^{+} &= (\oplus_{0\leq j\leq \frac{i-1}{2}}V_{\mathbb{Z}\gamma+\frac{i-2j}{2k}\gamma}\otimes M^{i,j})
\oplus(\oplus_{1\leq j\leq\frac{k-i-1}{2}}V_{\mathbb{Z}\gamma+\frac{-i-2j}{2k}\gamma}\otimes M^{i,i+j}),
\end{split}
\end{eqnarray*}

\begin{eqnarray*}
\begin{split}L(k,i)^{-} &= (\oplus_{0\leq j\leq \frac{i-1}{2}}V_{\mathbb{Z}\gamma+\frac{i-2j}{2k}\gamma}\otimes M^{i,j})
\oplus(\oplus_{1\leq j\leq\frac{k-i-1}{2}}V_{\mathbb{Z}\gamma+\frac{-i-2j}{2k}\gamma}\otimes M^{i,i+j}),
\end{split}
\end{eqnarray*}
that is, $L(k,i)^{+}\cong L(k,i)^{-}$ as $V_{\mathbb{Z}\gamma}^{+}\otimes K^{+}$-modules.

(3) If $i\in 2\mathbb{Z}+1, k\in 2\mathbb{Z}+1$,
\begin{eqnarray*}
\begin{split}L(k,i)^{+} &= (\oplus_{0\leq j\leq \frac{i-1}{2}}V_{\mathbb{Z}\gamma+\frac{i-2j}{2k}\gamma}\otimes M^{i,j})\oplus V_{\mathbb{Z}\gamma+\frac{1}{2}\gamma}^{+}\otimes (M^{k-i,\frac{k-i}{2}})^{+}\oplus V_{\mathbb{Z}\gamma+\frac{1}{2}\gamma}^{-}\otimes (M^{k-i,\frac{k-i}{2}})^{-}\\
&\oplus(\oplus_{1\leq j\leq\frac{k-i}{2}-1}V_{\mathbb{Z}\gamma+\frac{-i-2j}{2k}\gamma}\otimes M^{i,i+j}),
\end{split}
\end{eqnarray*}

\begin{eqnarray*}
\begin{split}L(k,i)^{-} &= (\oplus_{0\leq j\leq \frac{i-1}{2}}V_{\mathbb{Z}\gamma+\frac{i-2j}{2k}\gamma}\otimes M^{i,j})\oplus V_{\mathbb{Z}\gamma+\frac{1}{2}\gamma}^{+}\otimes (M^{k-i,\frac{k-i}{2}})^{-}\oplus V_{\mathbb{Z}\gamma+\frac{1}{2}\gamma}^{-}\otimes (M^{k-i,\frac{k-i}{2}})^{+}\\
&\oplus(\oplus_{1\leq j\leq\frac{k-i}{2}-1}V_{\mathbb{Z}\gamma+\frac{-i-2j}{2k}\gamma}\otimes M^{i,i+j}),
\end{split}
\end{eqnarray*}
as $V_{\mathbb{Z}\gamma}^{+}\otimes K^{+}$-modules.

(4)  If $i\in 2\mathbb{Z}, k\in 2\mathbb{Z}+1$,
\begin{eqnarray}\label{affineorbifold1}
\begin{split}L(k,i)^{+} &= (\oplus_{0\leq j\leq \frac{i}{2}-1}V_{\mathbb{Z}\gamma+\frac{i-2j}{2k}\gamma}\otimes M^{i,j})\oplus V_{\mathbb{Z}\gamma}^{+}\otimes (M^{i,\frac{i}{2}})^{+}\oplus V_{\mathbb{Z}\gamma}^{-}\otimes (M^{i,\frac{i}{2}})^{-}\\
&\oplus(\oplus_{1\leq j\leq\frac{k-i-1}{2}}V_{\mathbb{Z}\gamma+\frac{-i-2j}{2k}\gamma}\otimes M^{i,i+j}),
\end{split}
\end{eqnarray}

\begin{eqnarray}\label{affineorbifold2}
\begin{split}L(k,i)^{-} &= (\oplus_{0\leq j\leq \frac{i}{2}-1}V_{\mathbb{Z}\gamma+\frac{i-2j}{2k}\gamma}\otimes M^{i,j})\oplus V_{\mathbb{Z}\gamma}^{+}\otimes (M^{i,\frac{i}{2}})^{-}\oplus V_{\mathbb{Z}\gamma}^{-}\otimes (M^{i,\frac{i}{2}})^{+}\\
&\oplus(\oplus_{1\leq j\leq\frac{k-i-1}{2}}V_{\mathbb{Z}\gamma+\frac{-i-2j}{2k}\gamma}\otimes M^{i,i+j}),
\end{split}
\end{eqnarray}
as $V_{\mathbb{Z}\gamma}^{+}\otimes K^{+}$-modules.
\end{lem}
\begin{proof}
For $k\in 2\Z$,  from Lemma 3.3 and Lemma 3.5 of \cite{IJ}, where $L(k,i)^{\sigma_2,+}$ is our $L(k,i)^{+}$ and $K_0$ is our $K$ here, we have
$$
\begin{array}{ll}
	& L(k,0)^{+}=V_{\Z\gamma}^+\otimes K^+\oplus V_{\Z\gamma}^-\otimes K^-
	\oplus[\oplus_{\tiny{\begin{split}2\leq i<\frac{k}{2} \\  i\in 2\mathbb{Z} \ \  \end{split}}} (V_{\Z\gamma-\frac{i}{k}\gamma}\otimes M^{0,i}\oplus
	V_{\Z\gamma+\frac{i}{k}\gamma}\otimes M^{0,k-i})^+]\\&
	\\& \oplus_{\tiny{\begin{split}2\leq i<\frac{k}{2} \\  i\in 2\mathbb{Z} +1\ \  \end{split}}} (V_{\Z\gamma-\frac{i}{k}\gamma}\otimes M^{0,i}\oplus
	V_{\Z\gamma+\frac{i}{k}\gamma}\otimes M^{0,k-i})^+
	\oplus V_{\Z\gamma+\frac{1}{2}\gamma}^+\otimes  (M^{0,\frac{k}{2}})^+
	\oplus V_{\Z\gamma+\frac{1}{2}\gamma}^-\otimes  (M^{0,\frac{k}{2}})^-;	
\end{array}
$$
$$
\begin{array}{ll}
&L(k,0)^{-}=V_{\Z\gamma}^+\otimes K^-\oplus V_{\Z\gamma}^-\otimes K^+\oplus
[\oplus_{\tiny{\begin{split}2\leq i<\frac{k}{2} \\  i\in 2\mathbb{Z} \ \  \end{split}}} (V_{\Z\gamma-\frac{i}{k}\gamma}\otimes M^{0,i}\oplus
V_{\Z\gamma+\frac{i}{k}\gamma}\otimes M^{0,k-i})^-]
\\&\\&
\oplus_{\tiny{\begin{split}2\leq i<\frac{k}{2} \\  i\in 2\mathbb{Z} +1\ \  \end{split}}} (V_{\Z\gamma-\frac{i}{k}\gamma}\otimes M^{0,i}\oplus
V_{\Z\gamma+\frac{i}{k}\gamma}\otimes M^{0,k-i})^-\oplus V_{\Z\gamma+\frac{1}{2}\gamma}^+\otimes  (M^{0,\frac{k}{2}})^-
\oplus V_{\Z\gamma+\frac{1}{2}\gamma}^-\otimes  (M^{0,\frac{k}{2}})^+
\end{array}
$$
Note that for $i\in2\Z$,
$$
V_{\mathbb{Z}\gamma}^{+}\otimes (M^{i,\frac{i}{2}})^{+}\oplus V_{\mathbb{Z}\gamma}^{-}\otimes (M^{i,\frac{i}{2}})^{-}\subseteq L(k,i)^+.
$$
Then we have (1).

For $k\in2\Z$, $i\in 2\Z+1$, note that as $V_{\mathbb{Z}\gamma}^{+}\otimes K^{+}$-modules,
$$
V_{\mathbb{Z}\gamma+\frac{i-2j}{2k}\gamma}\otimes M^{i,j}\cong
V_{\mathbb{Z}\gamma+\frac{i-2(i-j)}{2k}\gamma}\otimes M^{i,i-j},
$$
for $0\leq j\leq \frac{i-1}{2}$. As $V_{\Z\gamma}^+$-modules, for $i+1\leq j\leq k-1$,
$$
V_{\Z\gamma+\frac{i-2j}{2k}\gamma}\cong V_{\Z\gamma+\frac{i-2(k-(j-i))}{2k}\gamma}.
$$
As $K^+$-modules, for $i+1\leq j\leq k-1$,
$$
M^{i,j}\cong M^{k-i,j-i}\cong M^{k-i,k-j}\cong M^{i,k-j+i}.
$$
Then for  $i+1\leq j\leq k-1$, as $V_{\Z\gamma}\otimes K^+$-modules,
$$
V_{\Z\gamma+\frac{i-2j}{2k}\gamma}\otimes M^{i,j}\cong V_{\Z\gamma+\frac{i-2(k-(j-i))}{2k}\gamma}\otimes
	 M^{i,k-j+i}.
$$
Thus (2) follows.

We now assume that $k\in 2\Z+1$, recall from \cite{DLY2},
$$
L(k,i)=\bigoplus_{j=0}^{k-1}V_{\mathbb{Z}\gamma+(i-2j)\gamma/2k}\otimes M^{i,j}.
$$
By \cite{DN} and \cite{JW1},  for $0\leq j\leq i$, $j\neq \frac{i}{2}$ (if $i\in 2\Z$),  $V_{\mathbb{Z}\gamma+\frac{i-2j}{2k}\gamma}\otimes M^{i,j}$ and $V_{\mathbb{Z}\gamma+\frac{i-2(i-j)}{2k}\gamma}\otimes M^{i,i-j}$ are irreducible and isomorphic $V_{\Z\gamma}^+\otimes K^+$-modules, and for $1\leq j\leq k-i-1$, $j\neq \frac{k-i}{2}$ (if $i\in 2\Z+1$) , $V_{\mathbb{Z}\gamma+\frac{i-2(i+j)}{2k}\gamma}\otimes M^{i,i+j}$ and $V_{\mathbb{Z}\gamma+\frac{i-2(k-j)}{2k}\gamma}\otimes M^{i,k-j}$ are are irreducible and isomorphic $V_{\Z\gamma}^+\otimes K^+$-modules.  If $i\in 2\Z$, $V_{\mathbb{Z}\gamma}^{+}\otimes (M^{i,\frac{i}{2}})^{+}\oplus V_{\mathbb{Z}\gamma}^{-}\otimes (M^{i,\frac{i}{2}})^{-}\subseteq L(k,i)^+$. If $i\in 2\Z+1$,  $V_{\mathbb{Z}\gamma+\frac{1}{2}\gamma}^{+}\otimes (M^{k-i,\frac{k-i}{2}})^{+}\oplus V_{\mathbb{Z}\gamma+\frac{1}{2}\gamma}^{-}\otimes (M^{k-i,\frac{k-i}{2}})^{-}\subseteq L(k,i)^+$. Then we get the desired decompositions.
\end{proof}

The following lemma can be obtained from Lemma 3.6, Lemma 3.9 and Lemma 3.10 of \cite{IJ} where $\overline{L(k,i)}^{\sigma_2,+}$ is our  $\overline{L(k,i)}^{+}$ here and note that $\overline{L(k,\frac{k}{2})}^{\sigma_2,+}=\overline{L(k,\frac{k}{2})}^{\sigma_2,1}+\overline{L(k,\frac{k}{2})}^{\sigma_2,2}$, $\overline{L(k,\frac{k}{2})}^{\sigma_2,-}=\overline{L(k,\frac{k}{2})}^{\sigma_2,3}+\overline{L(k,\frac{k}{2})}^{\sigma_2,4}$ in Lemma 3.10 of \cite{IJ}.
 \begin{lem}\label{twisted} Let $k$ be a positive integer and $0\leq i\leq k$, we have

(1) If $i\in 2\mathbb{Z}, i\neq\frac{k}{2}$,
\begin{eqnarray}\label{twistedmod1}
\begin{split}\overline{L(k,i)}^{+} &= V_{\mathbb{Z}\gamma}^{T_{1},+}\otimes W(k,i)^{+}\oplus V_{\mathbb{Z}\gamma}^{T_{1},-}\otimes W(k,i)^{-},
\end{split}
\end{eqnarray}

\begin{eqnarray}\label{twistedmod2}
\begin{split}\overline{L(k,i)}^{-} &= V_{\mathbb{Z}\gamma}^{T_{1},-}\otimes W(k,i)^{+}\oplus V_{\mathbb{Z}\gamma}^{T_{1},+}\otimes W(k,i)^{-},
\end{split}
\end{eqnarray}
as $V_{\mathbb{Z}\gamma}^{+}\otimes K^{+}$-modules.

(2) If $i\in 2\mathbb{Z}+1, i\neq\frac{k}{2}$,
\begin{eqnarray*}
\begin{split}\overline{L(k,i)}^{+} &= V_{\mathbb{Z}\gamma}^{T_{2},+}\otimes W(k,i)^{+}\oplus V_{\mathbb{Z}\gamma}^{T_{2},-}\otimes W(k,i)^{-},
\end{split}
\end{eqnarray*}

\begin{eqnarray*}
\begin{split}\overline{L(k,i)}^{-} &= V_{\mathbb{Z}\gamma}^{T_{2},-}\otimes W(k,i)^{+}\oplus V_{\mathbb{Z}\gamma}^{T_{2},+}\otimes W(k,i)^{-},
\end{split}
\end{eqnarray*}
as $V_{\mathbb{Z}\gamma}^{+}\otimes K^{+}$-modules.

(3) If $k\in 4\mathbb{Z}+2$,
\begin{eqnarray*}
\begin{split}\overline{L(k,\frac{k}{2})}^{+} &= V_{\mathbb{Z}\gamma}^{T_{2},+}\otimes (W(k,\frac{k}{2})^{+}\oplus\widetilde{W(k,\frac{k}{2}})^{-})\oplus V_{\mathbb{Z}\gamma}^{T_{2},-}\otimes (W(k,\frac{k}{2})^{-}\oplus\widetilde{W(k,\frac{k}{2}})^{+}),
\end{split}
\end{eqnarray*}

\begin{eqnarray*}
\begin{split}\overline{L(k,\frac{k}{2})}^{-} &= V_{\mathbb{Z}\gamma}^{T_{2},-}\otimes (W(k,\frac{k}{2})^{+}\oplus\widetilde{W(k,\frac{k}{2}})^{-})\oplus V_{\mathbb{Z}\gamma}^{T_{2},+}\otimes (W(k,\frac{k}{2})^{-}\oplus\widetilde{W(k,\frac{k}{2}})^{+}),
\end{split}
\end{eqnarray*}
as $V_{\mathbb{Z}\gamma}^{+}\otimes K^{+}$-modules.

(4) If $k\in 4\mathbb{Z}$,
\begin{eqnarray*}
\begin{split}\overline{L(k,\frac{k}{2})}^{+} &= V_{\mathbb{Z}\gamma}^{T_{1},+}\otimes (W(k,\frac{k}{2})^{+}\oplus\widetilde{W(k,\frac{k}{2}})^{-})\oplus V_{\mathbb{Z}\gamma}^{T_{1},-}\otimes (W(k,\frac{k}{2})^{-}\oplus\widetilde{W(k,\frac{k}{2}})^{+}),
\end{split}
\end{eqnarray*}
\begin{eqnarray*}
\begin{split}\overline{L(k,\frac{k}{2})}^{-} &= V_{\mathbb{Z}\gamma}^{T_{1},-}\otimes (W(k,\frac{k}{2})^{+}\oplus\widetilde{W(k,\frac{k}{2}})^{-})\oplus V_{\mathbb{Z}\gamma}^{T_{1},+}\otimes (W(k,\frac{k}{2})^{-}\oplus\widetilde{W(k,\frac{k}{2}})^{+}),
\end{split}
\end{eqnarray*}
as $V_{\mathbb{Z}\gamma}^{+}\otimes K^{+}$-modules.

\end{lem}

\section{Classification of the irreducible modules of $K(\g, k)^{\sigma}$
}\label{Sect:affine case}\def\theequation{3.\arabic{equation}}

In this section, let $K(\g,k)^{\sigma}$ be the orbifold vertex operator subalgebra of the parafermion vertex operator algebra $K(\g,k)$ associated to Lie superalgebra $\g=osp(1|2)$, i.e., the fixed-point subalgebra of $K(\g,k)$ under $\sigma$. We then classify the irreducible modules for $K(\g,k)^{\sigma}$.

Let $\left(V,Y,1,\omega\right)$ be a vertex operator algebra (see
\cite{FLM}, \cite{LL}) and $g$ an automorphism of $V$ with finite order $T$.  Let $W\left\{ z\right\} $
denote the space of $W$-valued formal series in arbitrary complex
powers of $z$ for a vector space $W$. Denote the decomposition of $V$ into eigenspaces with respect to the action of $g$ by $$V=\bigoplus_{r\in\Z}V^{r},$$
where $V^{r}=\{v\in V|\ gv=e^{-\frac{2\pi ir}{T}}v\},$ $i=\sqrt{-1}$.

\begin{definition}A \emph{weak $g$-twisted $V$-module} $M$ is
a vector space with a linear map
\[
Y_{M}:V\to\left(\text{End}M\right)\{z\}
\]

\[
v\mapsto Y_{M}\left(v,z\right)=\sum_{n\in\mathbb{Q}}v_{n}z^{-n-1}\ \left(v_{n}\in\mbox{End}M\right)
\]
which satisfies the following conditions for $0\leq r\leq T-1$, $u\in V^{r}\ ,v\in V, w\in M$:

\[ Y_{M}\left(u,z\right)=\sum_{n\in\frac{r}{T}+\mathbb{Z}}u_{n}z^{-n-1}
\]

\[
u_{n}w=0\ {\rm for} \ n\gg0,
\]

\[
Y_{M}\left(\mathbf{1},z\right)=Id_{M},
\]

\[
z_{0}^{-1}\text{\ensuremath{\delta}}\left(\frac{z_{1}-z_{2}}{z_{0}}\right)Y_{M}\left(u,z_{1}\right)
Y_{M}\left(v,z_{2}\right)-z_{0}^{-1}\delta\left(\frac{z_{2}-z_{1}}{-z_{0}}\right)Y_{M}\left(v,z_{2}\right)Y_{M}\left(u,z_{1}\right)
\]

\[
=z_{1}^{-1}\left(\frac{z_{2}+z_{0}}{z_{1}}\right)^{\frac{r}{T}}\delta\left(\frac{z_{2}+z_{0}}{z_{1}}\right)Y_{M}\left(Y\left(u,z_{0}\right)v,z_{2}\right),
\]
 where $\delta\left(z\right)=\sum_{n\in\mathbb{Z}}z^{n}$. \end{definition}
 The following identities are the consequences of the twisted-Jacobi identity \cite{DLM3} (see also \cite{Ab}, \cite{DJ}).
 \begin{eqnarray}[u_{m+\frac{r}{T}},v_{n+\frac{s}{T}}]=\sum_{i=0}^{\infty}\binom{m+\frac{r}{T}}{i} (u_{i}v)_{m+n+\frac{r+s}{T}-i},\label{eq:3.1.}\end{eqnarray}
 \begin{eqnarray}\sum_{i\geq 0}\binom{\frac{r}{T}}{i}(u_{m+i}v)_{n+\frac{r+s}{T}-i}=\sum_{i\geq 0}(-1)^{i}\binom{m}{i}(u_{m+\frac{r}{T}-i}v_{n+\frac{s}{T}+i}-(-1)^{m}v_{m+n+\frac{s}{T}-i}u_{\frac{r}{T}+i}),\label{eq:3.2.}\end{eqnarray}
 where $u\in V^{r}, \ v\in V^{s},\ m,n\in \Z$.

\begin{definition}

A \emph{$g$-twisted $V$-module} is a weak $g$-twisted $V$-module\emph{
}$M$ which carries a $\mathbb{C}$-grading $M=\bigoplus_{\lambda\in\mathbb{C}}M_{\lambda},$
where $M_{\lambda}=\{w\in M|L(0)w=\lambda w\}$ and $L(0)$ is one of the coefficient operators of $Y(\omega,z)=\sum_{n\in\mathbb{Z}}L(n)z^{-n-2}.$
Moreover we require
that $\dim M_{\lambda}$ is finite and for fixed $\lambda,$ $M_{\lambda+\frac{n}{T}}=0$
for all small enough integers $n.$

\end{definition}

\begin{definition}An \emph{admissible $g$-twisted $V$-module} $M=\oplus_{n\in\frac{1}{T}\mathbb{Z}_{+}}M\left(n\right)$
is a $\frac{1}{T}\mathbb{Z}_{+}$-graded weak $g$-twisted module
such that $u_{m}M\left(n\right)\subset M\left(\mbox{wt}u-m-1+n\right)$
for homogeneous $u\in V$ and $m,n\in\frac{1}{T}\mathbb{Z}.$ $ $

\end{definition}

If $g=Id_{V}$, we have the notions of weak, ordinary and admissible
$V$-modules \cite{DLM3}.

  Recall that $L(k,i)$ for $0\leq i\leq k$ are the irreducible modules for the rational vertex operator algebra $L_{\hat{sl_{2}}}(k,0)$ with the top level $U^{i}=\bigoplus_{j=0}^{i}\mathbb{C}v^{i,j}$ an $(i+1)$-dimensional irreducible module for $\C h(0)\oplus \C e(0)\oplus \C f(0)\cong sl_2$.  From \cite{CFK}, we know that the vertex operator superalgebra $L_{\hat{g}}(k,0)=L^{even}_{\hat{g}}(k,0)\oplus L^{odd}_{\hat{g}}(k,0)$, where $L^{even}_{\hat{g}}(k,0)$ is its even subalgebra which is simple and $L^{odd}_{\hat{g}}(k,0)$ is an order two simple current with $$L^{even}_{\hat{g}}(k,0)=\bigoplus_{i=0, \ i\ even}^{k}L(k,i)\otimes V_{i+1,1} $$and
  $$L^{odd}_{\hat{g}}(k,0)=\bigoplus_{i=0, \ i\ odd}^{k}L(k,i)\otimes V_{i+1,1}.$$
From the definition of the automorphism $\sigma$, we see that $L_{\hat{g}}(k,0)^{\sigma}=L^{even}_{\hat{g}}(k,0)^{\sigma}$, and $\sigma$ is the automorphism of vertex operator algebra $L^{even}_{\hat{g}}(k,0)$ with order $2$. Note that $L_{\hat{sl_{2}}}(k,0)^{\sigma}$ is the subalgebra of $L^{even}_{\hat{g}}(k,0)^{\sigma}$. The irreducible modules of $L_{\hat{sl_{2}}}(k,0)^{\sigma}$ are classified in \cite{JW2}.

\begin{thm}\cite{JW2}\label{thm:affine sl2}
There are $4k+4$ inequivalent irreducible modules of $L_{\hat{sl_{2}}}(k,0)^{\sigma}$.
  More precisely, the set
 \begin{eqnarray*}
&& \{ L(k,i)^{+}, L(k,i)^{-}, \overline{L(k,i)}^{+}, \overline{L(k,i)}^{-} \ \mbox{for} \ 0\leq i\leq k\} \\
 \end{eqnarray*} gives all inequivalent irreducible $L_{\hat{sl_{2}}}(k,0)^{\sigma}$-modules.
\end{thm}

\begin{rmk}\label{rmk:affine sl2} With the notations in  Theorem \ref{thm:affine sl2}, we call $L(k,i)^{+}, L(k,i)^{-}$ untwisted type modules and  $\overline{L(k,i)}^{+}, \overline{L(k,i)}^{-} $ twisted type modules respectively.

\end{rmk}

Let $\{ h, e, f\}$
be a standard Chevalley basis of $sl_2$ with brackets $[h,e] = 2e$, $[h,f] = -2f$,
$[e,f] = h$. Set $$h^{'}=e+f, \ e^{'}=\frac{1}{2}(h-e+f),\ f^{'}=\frac{1}{2}(h+e-f).$$ Then $\{h^{'}, e^{'}, f^{'}\}$ is a $sl_2$-triple. Let $h^{''}=\frac{1}{4}h^{'}=\frac{1}{4}(e+f)$, from \cite{L2}, we know that $e^{2\pi \sqrt{-1}h^{''}(0)}$ is an automorphism of $L_{\hat{g}}(k,0)$ and
$$e^{2\pi \sqrt{-1}h^{''}(0)}(h^{'})=h^{'},\ e^{2\pi \sqrt{-1}h^{''}(0)}(e^{'})=-e^{'},\ e^{2\pi \sqrt{-1}h^{''}(0)}(f^{'})=-f^{'},$$
$$e^{2\pi \sqrt{-1}h^{''}(0)}(x+y)=-\sqrt{-1}(x+y),\ e^{2\pi \sqrt{-1}h^{''}(0)}(x-y)=\sqrt{-1}(x-y).$$
Thus $e^{2\pi \sqrt{-1}h^{''}(0)}=\sigma$.

Now we recall the classification of the modules of the vertex operator algebra $L_{\hat{g}}(k,0)^{\sigma}$\cite{JW4}.

 \begin{thm}\cite{JW4}\label{thm:affinetwist'} The modules
 \begin{eqnarray}\label{mod1.}
 M_{r}^{even,+}=\bigoplus_{i=0, \ i\in 4\mathbb{Z}}^{k}L(k,i)^{+}\otimes V_{i+1,r}\bigoplus_{i=0, \ i\in 4\mathbb{Z}+2}^{k}L(k,i)^{-}\otimes V_{i+1,r}\end{eqnarray}
  \begin{eqnarray}\label{mod2.}
 M_{r}^{even,-}=\bigoplus_{i=0, \ i\in 4\mathbb{Z}}^{k}L(k,i)^{-}\otimes V_{i+1,r}\bigoplus_{i=0, \ i\in 4\mathbb{Z}+2}^{k}L(k,i)^{+}\otimes V_{i+1,r}\end{eqnarray}
 \begin{eqnarray}\label{mod3.}
 \overline{M_{r}^{even,+}}=\bigoplus_{i=0, \ i\in 4\mathbb{Z}}^{k}\overline{L(k,i)}^{+}\otimes V_{i+1,r}\bigoplus_{i=0, \ i\in 4\mathbb{Z}+2}^{k}\overline{L(k,i)}^{-}\otimes V_{i+1,r}\end{eqnarray}
 \begin{eqnarray}\label{mod4.}
 \overline{M_{r}^{even,-}}=\bigoplus_{i=0, \ i\in 4\mathbb{Z}}^{k}\overline{L(k,i)}^{-}\otimes V_{i+1,r}\bigoplus_{i=0, \ i\in 4\mathbb{Z}+2}^{k}\overline{L(k,i)}^{+}\otimes V_{i+1,r}\end{eqnarray}

\begin{eqnarray}\label{mod5.}M_{r}^{odd,1}=\bigoplus_{i=0, \ i\in 4\mathbb{Z}+1}^{k}L(k,i)^{+}\otimes V_{i+1,r}\bigoplus_{i=0, \ i\in 4\mathbb{Z}+3}^{k}L(k,i)^{-}\otimes V_{i+1,r}
  \end{eqnarray}
\begin{eqnarray}\label{mod6.}M_{r}^{odd,2}=\bigoplus_{i=0, \ i\in 4\mathbb{Z}+1}^{k}L(k,i)^{-}\otimes V_{i+1,r}\bigoplus_{i=0, \ i\in 4\mathbb{Z}+3}^{k}L(k,i)^{+}\otimes V_{i+1,r}
  \end{eqnarray}

\begin{eqnarray}\label{mod7.}\overline{M_{r}^{odd,1}}=\bigoplus_{i=0, \ i\in 4\mathbb{Z}+1}^{k}\overline{L(k,i)}^{+}\otimes V_{i+1,r}\bigoplus_{i=0, \ i\in 4\mathbb{Z}+3}^{k}\overline{L(k,i)}^{-}\otimes V_{i+1,r}
  \end{eqnarray}

 \begin{eqnarray}\label{mod8.}\overline{M_{r}^{odd,2}}=\bigoplus_{i=0, \ i\in 4\mathbb{Z}+1}^{k}\overline{L(k,i)}^{-}\otimes V_{i+1,r}\bigoplus_{i=0, \ i\in 4\mathbb{Z}+3}^{k}\overline{L(k,i)}^{+}\otimes V_{i+1,r}
  \end{eqnarray}

  for $1\leq r\leq 2k+2$ form a complete list of inequivalent simple modules of the vertex operator algebra $L_{\hat{g}}(k,0)^{\sigma}$.

 \end{thm}

 \begin{rmk}\label{rmk:affine orbifold} With the notations in Theorem \ref{thm:affinetwist'}, we call $M_{r}^{even,\pm}, M_{r}^{odd,1}, M_{r}^{odd,2}$ untwisted type modules and  $\overline{M_{r}^{even,\pm}}, \overline{M_{r}^{odd,1}}, \overline{M_{r}^{odd,2}}$ twisted type modules respectively.
\end{rmk}

 For simplicity, we denote untwisted type modules $M_{r}^{odd,1}$ by $M_{r}^{odd,+}$ and $M_{r}^{odd,2}$ by $M_{r}^{odd,-}$. We denote twisted type modules $\overline{M_{r}^{odd,1}}$ by $\overline{M_{r}^{odd,+}}$ and $\overline{M_{r}^{odd,2}}$ by $\overline{M_{r}^{odd,-}}$.

 \begin{thm}\label{thm:irr para}
  (1) If $k=2n+1$, $n\geq 1$, there are $4(k+1)(n+4)$ inequivalent irreducible modules of $K(\g,k)^{\sigma}$, the list is

  \begin{eqnarray}\label{untwistedmod1.}
 U_{r}^{even,k,s}=\sum_{\tiny{\begin{split}i^{'}=2\ \ \  \\ i^{'}\in 2\mathbb{Z}\ \ \  \\ \ \ \ \ i^{'}\geq i-2j  \end{split}}}^{k}M^{i^{'},\frac{i^{'}-i+2j}{2}}\otimes V_{i^{'}+1,r}+\sum_{\tiny{\begin{split}i^{''}=0\ \ \ \ \   \\ i^{''}\in 2\mathbb{Z}\ \ \ \ \ \\  i^{''}\leq i-2j-2 \\ i-2j\leq k-1 \end{split}}}^{k}M^{i^{''},\frac{i^{''}+i-2j}{2}}\otimes V_{i^{''}+1,r}\end{eqnarray}
 for $2\leq i\leq k, i\in 2\mathbb{Z}, 0\leq j\leq \frac{i}{2}-1$ and $s=i-2j=2,4,\cdots,2n$, $1\leq r\leq 2k+2$.

  \begin{eqnarray}\label{untwistedmod2.}
 S^{even,k,1}_{r}=\bigoplus_{\tiny{\begin{split} i=0 \ \ \ \\ i\in 4\mathbb{Z} \ \ \end{split}}}^{k}(M^{i,\frac{i}{2}})^{+}\otimes V_{i+1,r}\oplus\bigoplus_{\tiny{\begin{split}i=0\ \ \  \\ i\in 4\mathbb{Z}+2 \end{split}}}^{k}(M^{i,\frac{i}{2}})^{-}\otimes V_{i+1,r}\end{eqnarray}
for $1\leq r\leq 2k+2$.
 \begin{eqnarray}\label{untwistedmod2'.}
 S^{even,k,2}_{r}=\bigoplus_{\tiny{\begin{split} i=0 \ \ \ \\ i\in 4\mathbb{Z} \ \ \end{split}}}^{k}(M^{i,\frac{i}{2}})^{-}\otimes V_{i+1,r}\oplus\bigoplus_{\tiny{\begin{split}i=0\ \ \  \\ i\in 4\mathbb{Z}+2 \end{split}}}^{k}(M^{i,\frac{i}{2}})^{+}\otimes V_{i+1,r}\end{eqnarray}
for $1\leq r\leq 2k+2$.
 \begin{eqnarray}\label{untwistedmod3.}
 U_{r}^{odd,k,s}=\sum_{\tiny{\begin{split}i^{'}=1\ \ \  \\ \ \ \ \ i^{'}\in 2\mathbb{Z}+1 \\ \ \ \ \   i^{'}\geq i-2j  \end{split}}}^{k}M^{i^{'},\frac{i^{'}-i+2j}{2}}\otimes V_{i^{'}+1,r}+\sum_{\tiny{\begin{split}i^{''}=0\ \ \ \ \   \\ \ \ \ i^{''}\in 2\mathbb{Z}+1\ \ \ \ \ \\ \ \ \  i^{''}\leq i-2j \leq k \end{split}}}^{k}M^{i^{''},\frac{i^{''}+i-2j}{2}}\otimes V_{i^{''}+1,r}\end{eqnarray}
 for $1\leq i\leq k, i\in 2\mathbb{Z}+1, 0\leq j\leq \frac{i-1}{2}, (i,j)\neq(k,0)$ and $s=i-2j=1,3,\cdots,2n+1$, $1\leq r\leq 2k+2$.

  \begin{eqnarray}\label{untwistedmod4.}
 S^{odd,k,1}_{r}=\bigoplus_{\tiny{\begin{split} i=0 \ \ \ \ \ \\ i\in 4\mathbb{Z}+1 \ \ \end{split}}}^{k}(M^{k-i,\frac{k-i}{2}})^{+}\otimes V_{i+1,r}\oplus\bigoplus_{\tiny{\begin{split}i=0\ \ \  \\ i\in 4\mathbb{Z}+3 \end{split}}}^{k}(M^{k-i,\frac{k-i}{2}})^{-}\otimes V_{i+1,r}\end{eqnarray}
for $1\leq r\leq 2k+2$.
 \begin{eqnarray}\label{untwistedmod4'.}
 S^{odd,k,2}_{r}=\bigoplus_{\tiny{\begin{split} i=0 \ \ \ \ \ \\ i\in 4\mathbb{Z}+1 \ \ \end{split}}}^{k}(M^{k-i,\frac{k-i}{2}})^{-}\otimes V_{i+1,r}\oplus\bigoplus_{\tiny{\begin{split}i=0\ \ \  \\ i\in 4\mathbb{Z}+3 \end{split}}}^{k}(M^{k-i,\frac{k-i}{2}})^{+}\otimes V_{i+1,r}\end{eqnarray}
 for $1\leq r\leq 2k+2$.
  \begin{eqnarray}\label{twistedmod5.}
 \overline{U^{even,k,1}_{r}}=\bigoplus_{\tiny{\begin{split} i=0 \ \ \ \\ i\in 4\mathbb{Z} \ \ \end{split}}}^{k}W(k,i)^{+}\otimes V_{i+1,r}\oplus\bigoplus_{\tiny{\begin{split}i=0\ \ \  \\ i\in 4\mathbb{Z}+2 \end{split}}}^{k}W(k,i)^{-}\otimes V_{i+1,r}\end{eqnarray}
for $1\leq r\leq 2k+2$.
 \begin{eqnarray}\label{twistedmod5'.}
  \overline{U^{even,k,2}_{r}}=\bigoplus_{\tiny{\begin{split} i=0 \ \ \ \\ i\in 4\mathbb{Z} \ \ \end{split}}}^{k}W(k,i)^{-}\otimes V_{i+1,r}\oplus\bigoplus_{\tiny{\begin{split}i=0\ \ \  \\ i\in 4\mathbb{Z}+2 \end{split}}}^{k}W(k,i)^{+}\otimes V_{i+1,r}\end{eqnarray}
 for $1\leq r\leq 2k+2$.
  \begin{eqnarray}\label{twistedmod6.}
 \overline{U^{odd,k,1}_{r}}=\bigoplus_{\tiny{\begin{split} i=0 \ \ \ \\ i\in 4\mathbb{Z}+1 \ \end{split}}}^{k}W(k,i)^{+}\otimes V_{i+1,r}\oplus\bigoplus_{\tiny{\begin{split}i=0\ \ \  \\ i\in 4\mathbb{Z}+3 \end{split}}}^{k}W(k,i)^{-}\otimes V_{i+1,r}\end{eqnarray}
for $1\leq r\leq 2k+2$.
 \begin{eqnarray}\label{twistedmod6'.}
  \overline{U^{odd,k,2}_{r}}=\bigoplus_{\tiny{\begin{split} i=0 \ \ \ \\ i\in 4\mathbb{Z}+1  \ \end{split}}}^{k}W(k,i)^{-}\otimes V_{i+1,r}\oplus\bigoplus_{\tiny{\begin{split}i=0\ \ \  \\ i\in 4\mathbb{Z}+3 \end{split}}}^{k}W(k,i)^{+}\otimes V_{i+1,r}\end{eqnarray}
 for $1\leq r\leq 2k+2$.

(2) If $k=4n$, $n\geq 1$, there are $2(k+1)(4n+3)$ inequivalent irreducible modules of $K(\g,k)^{\sigma}$, the list is

  \begin{eqnarray}\label{untwistedmod7.}
 U_{r}^{even,k,s}=\sum_{\tiny{\begin{split}i^{'}=2\ \ \  \\ i^{'}\in 2\mathbb{Z}\ \ \  \\ \ \ \ \ i^{'}\geq i-2j  \end{split}}}^{k}M^{i^{'},\frac{i^{'}-i+2j}{2}}\otimes V_{i^{'}+1,r}+\sum_{\tiny{\begin{split}i^{''}=0\ \ \ \ \   \\ i^{''}\in 2\mathbb{Z}\ \ \ \ \ \\  i^{''}\leq i-2j-2 \\ i-2j\leq k-2 \end{split}}}^{k}M^{i^{''},\frac{i^{''}+i-2j}{2}}\otimes V_{i^{''}+1,r}\end{eqnarray}
 for $2\leq i\leq k, i\in 2\mathbb{Z}, 0\leq j\leq \frac{i}{2}-1, (i,j)\neq(k,0)$ and $s=i-2j=2,4,6\cdots,2(2n-1)$, $1\leq r\leq 2k+2$.

 \begin{eqnarray}\label{untwistedmod7'.}
 U_{r}^{odd,k,s}=\sum_{\tiny{\begin{split}i^{'}=1\ \ \  \\ \ \ \ \ i^{'}\in 2\mathbb{Z}+1 \\ \ \ \ \   i^{'}\geq i-2j  \end{split}}}^{k}M^{i^{'},\frac{i^{'}-i+2j}{2}}\otimes V_{i^{'}+1,r}+\sum_{\tiny{\begin{split}i^{''}=0\ \ \ \ \   \\ \ \ \ i^{''}\in 2\mathbb{Z}+1\ \ \ \ \ \\ \ \ \  i^{''}\leq i-2j \leq k-1 \end{split}}}^{k}M^{i^{''},\frac{i^{''}+i-2j}{2}}\otimes V_{i^{''}+1,r}\end{eqnarray}
 for $1\leq i\leq k, i\in 2\mathbb{Z}+1, 0\leq j\leq \frac{i-1}{2},$ and $s=i-2j=1,3,\cdots,2(2n)-1$, $1\leq r\leq 2k+2$.

  \begin{eqnarray}\label{untwistedmod8.}
 S^{even,k,1}_{r}=\bigoplus_{\tiny{\begin{split} i=0 \ \ \ \\ i\in 4\mathbb{Z} \ \ \end{split}}}^{k}(M^{i,\frac{i}{2}})^{+}\otimes V_{i+1,r}\oplus\bigoplus_{\tiny{\begin{split}i=0\ \ \  \\ i\in 4\mathbb{Z}+2 \end{split}}}^{k}(M^{i,\frac{i}{2}})^{-}\otimes V_{i+1,r}.\end{eqnarray}

 \begin{eqnarray}\label{untwistedmod8'.}
 S^{even,k,2}_{r}=\bigoplus_{\tiny{\begin{split} i=0 \ \ \ \\ i\in 4\mathbb{Z} \ \ \end{split}}}^{k}(M^{i,\frac{i}{2}})^{-}\otimes V_{i+1,r}\oplus\bigoplus_{\tiny{\begin{split}i=0\ \ \  \\ i\in 4\mathbb{Z}+2 \end{split}}}^{k}(M^{i,\frac{i}{2}})^{+}\otimes V_{i+1,r}.\end{eqnarray}
 for $1\leq r\leq 2k+2$.

 \begin{eqnarray}\label{untwistedmod10.}
 \overline{U^{odd,k}_{r}}=\bigoplus_{\tiny{\begin{split} i=0 \ \ \ \\ i\in 4\mathbb{Z}+1 \ \end{split}}}^{k}W(k,i)^{+}\otimes V_{i+1,r}\oplus\bigoplus_{\tiny{\begin{split}i=0\ \ \  \\ i\in 4\mathbb{Z}+3 \end{split}}}^{k}W(k,i)^{-}\otimes V_{i+1,r}\end{eqnarray}
for $1\leq r\leq 2k+2$.

 (i) If $\frac{k}{2}\in 4\mathbb{Z}$,

 \begin{eqnarray}\label{untwistedmod9.}
 \overline{U^{even,k,1}_{r}}=\bigoplus_{\tiny{\begin{split} i=0 \ \ \ \\ i\in 4\mathbb{Z} \ \ \\ i\neq \frac{k}{2} \ \ \end{split}}}^{k}W(k,i)^{+}\otimes V_{i+1,r}\oplus\big(W(k,\frac{k}{2})^{+}\oplus \widetilde{W(k,\frac{k}{2}}^{-})\big)\otimes V_{\frac{k}{2}+1,r}\oplus\bigoplus_{\tiny{\begin{split}i=0\ \ \  \\ i\in 4\mathbb{Z}+2 \end{split}}}^{k}W(k,i)^{-}\otimes V_{i+1,r}.\end{eqnarray}
 for $1\leq r\leq k+1$.

 \begin{eqnarray}\label{untwistedmod9'.}
   \overline{U^{even,k,2}_{r}}=\bigoplus_{\tiny{\begin{split} i=0 \ \ \ \\ i\in 4\mathbb{Z} \ \ \\ i\neq \frac{k}{2} \ \ \end{split}}}^{k}W(k,i)^{-}\otimes V_{i+1,r}\oplus\big(W(k,\frac{k}{2})^{-}\oplus \widetilde{W(k,\frac{k}{2}}^{+})\big)\otimes V_{\frac{k}{2}+1,r}\oplus\bigoplus_{\tiny{\begin{split}i=0\ \ \  \\ i\in 4\mathbb{Z}+2 \end{split}}}^{k}W(k,i)^{+}\otimes V_{i+1,r}.\end{eqnarray}
 for $1\leq r\leq k+1$.

  (ii) If $\frac{k}{2}\in 4\mathbb{Z}+2$,

 \begin{eqnarray}\label{untwistedmod11.}
 \overline{U^{even,k,1}_{r}}=\bigoplus_{\tiny{\begin{split} i=0 \ \ \ \\ i\in 4\mathbb{Z} \ \  \end{split}}}^{k}W(k,i)^{+}\otimes V_{i+1,r}\oplus\big(W(k,\frac{k}{2})^{+}\oplus \widetilde{W(k,\frac{k}{2}}^{-})\big)\otimes V_{\frac{k}{2}+1,r}\oplus\bigoplus_{\tiny{\begin{split}i=0\ \ \  \\ i\in 4\mathbb{Z}+2 \\ i\neq \frac{k}{2} \ \ \end{split}}}^{k}W(k,i)^{-}\otimes V_{i+1,r}.\end{eqnarray}
 for $1\leq r\leq k+1$.

 \begin{eqnarray}\label{untwistedmod12'.}
   \overline{U^{even,k,2}_{r}}=\bigoplus_{\tiny{\begin{split} i=0 \ \ \ \\ i\in 4\mathbb{Z} \ \ \end{split}}}^{k}W(k,i)^{-}\otimes V_{i+1,r}\oplus\big(W(k,\frac{k}{2})^{-}\oplus \widetilde{W(k,\frac{k}{2}}^{+})\big)\otimes V_{\frac{k}{2}+1,r}\oplus\bigoplus_{\tiny{\begin{split}i=0\ \ \  \\ i\in 4\mathbb{Z}+2 \\ i\neq \frac{k}{2} \ \ \end{split}}}^{k}W(k,i)^{+}\otimes V_{i+1,r}.\end{eqnarray}
 for $1\leq r\leq k+1$.

(3) If $k=4n+2$, $n\geq 1$, there are $2(k+1)(4n+7)$ inequivalent irreducible modules of $K(\g,k)^{\sigma}$, the list is

  \begin{eqnarray}\label{untwistedmod13.}
 U_{r}^{even,k,s}=\sum_{\tiny{\begin{split}i^{'}=2\ \ \  \\ i^{'}\in 2\mathbb{Z}\ \ \  \\ \ \ \ \ i^{'}\geq i-2j  \end{split}}}^{k}M^{i^{'},\frac{i^{'}-i+2j}{2}}\otimes V_{i^{'}+1,r}+\sum_{\tiny{\begin{split}i^{''}=0\ \ \ \ \   \\ i^{''}\in 2\mathbb{Z}\ \ \ \ \ \\  i^{''}\leq i-2j-2 \\ i-2j\leq k-2 \end{split}}}^{k}M^{i^{''},\frac{i^{''}+i-2j}{2}}\otimes V_{i^{''}+1,r}\end{eqnarray}
 for $2\leq i\leq k, i\in 2\mathbb{Z}, 0\leq j\leq \frac{i}{2}-1, (i,j)\neq(k,0)$ and $s=i-2j=2,4,\cdots,2(2n)$, $1\leq r\leq 2k+2$.
\begin{eqnarray}\label{untwistedmod13'.}
 U_{r}^{odd,k,s}=\sum_{\tiny{\begin{split}i^{'}=1\ \ \  \\ \ \ \ \ i^{'}\in 2\mathbb{Z}+1 \\ \ \ \ \   i^{'}\geq i-2j  \end{split}}}^{k}M^{i^{'},\frac{i^{'}-i+2j}{2}}\otimes V_{i^{'}+1,r}+\sum_{\tiny{\begin{split}i^{''}=0\ \ \ \ \   \\ \ \ \ i^{''}\in 2\mathbb{Z}+1\ \ \ \ \ \\ \ \ \  i^{''}\leq i-2j \leq k-1 \end{split}}}^{k}M^{i^{''},\frac{i^{''}+i-2j}{2}}\otimes V_{i^{''}+1,r}\end{eqnarray}
 for $1\leq i\leq k, i\in 2\mathbb{Z}+1, 0\leq j\leq \frac{i-1}{2},$ and $s=i-2j=1,3,5\cdots,2(2n+1)-1$, $1\leq r\leq 2k+2$.

  \begin{eqnarray}\label{untwistedmod14.}
 S^{even,k,1}_{r}=\bigoplus_{\tiny{\begin{split} i=0 \ \ \ \\ i\in 4\mathbb{Z} \ \ \end{split}}}^{k}(M^{i,\frac{i}{2}})^{+}\otimes V_{i+1,r}\oplus\bigoplus_{\tiny{\begin{split}i=0\ \ \  \\ i\in 4\mathbb{Z}+2 \end{split}}}^{k}(M^{i,\frac{i}{2}})^{-}\otimes V_{i+1,r}\end{eqnarray}
 for $1\leq r\leq 2k+2$.

 \begin{eqnarray}\label{untwistedmod14'.}
 S^{even,k,2}_{r}=\bigoplus_{\tiny{\begin{split} i=0 \ \ \ \\ i\in 4\mathbb{Z} \ \ \end{split}}}^{k}(M^{i,\frac{i}{2}})^{-}\otimes V_{i+1,r}\oplus\bigoplus_{\tiny{\begin{split}i=0\ \ \  \\ i\in 4\mathbb{Z}+2 \end{split}}}^{k}(M^{i,\frac{i}{2}})^{+}\otimes V_{i+1,r}\end{eqnarray}
 for $1\leq r\leq 2k+2$.

  \begin{eqnarray}\label{untwistedmod15.}
 S^{even,k,3}_{r}=\bigoplus_{\tiny{\begin{split} i=0 \ \ \ \\ i\in 4\mathbb{Z} \ \ \end{split}}}^{k}(M^{k-i,\frac{k-i}{2}})^{+}\otimes V_{i+1,r}\oplus\bigoplus_{\tiny{\begin{split}i=0\ \ \  \\ i\in 4\mathbb{Z}+2 \end{split}}}^{k}(M^{k-i,\frac{k-i}{2}})^{-}\otimes V_{i+1,r}\end{eqnarray}
 for $1\leq r\leq 2k+2$.
 \begin{eqnarray}\label{untwistedmod16'.}
 S^{even,k,4}_{r}=\bigoplus_{\tiny{\begin{split} i=0 \ \ \ \\ i\in 4\mathbb{Z} \ \ \end{split}}}^{k}(M^{k-i,\frac{k-i}{2}})^{-}\otimes V_{i+1,r}\oplus\bigoplus_{\tiny{\begin{split}i=0\ \ \  \\ i\in 4\mathbb{Z}+2 \end{split}}}^{k}(M^{k-i,\frac{k-i}{2}})^{+}\otimes V_{i+1,r}.\end{eqnarray}
 for $1\leq r\leq 2k+2$.

 \begin{eqnarray}\label{untwistedmod17.}
 \overline{U^{even,k}_{r}}=\bigoplus_{\tiny{\begin{split} i=0 \ \ \ \\ i\in 4\mathbb{Z} \ \end{split}}}^{k}W(k,i)^{+}\otimes V_{i+1,r}\oplus\bigoplus_{\tiny{\begin{split}i=0\ \ \  \\ i\in 4\mathbb{Z}+2 \end{split}}}^{k}W(k,i)^{-}\otimes V_{i+1,r}\end{eqnarray}
for $1\leq r\leq 2k+2$.

 (i) If $\frac{k}{2}\in 4\mathbb{Z}+1$,

 \begin{eqnarray*}
 \overline{U^{odd,k,1}_{r}}=\bigoplus_{\tiny{\begin{split} i=0 \ \ \ \\ i\in 4\mathbb{Z}+1 \ \ \\ i\neq \frac{k}{2} \ \ \end{split}}}^{k}W(k,i)^{+}\otimes V_{i+1,r}\oplus\big(W(k,\frac{k}{2})^{+}\oplus \widetilde{W(k,\frac{k}{2}}^{-})\big)\otimes V_{\frac{k}{2}+1,r}\oplus\bigoplus_{\tiny{\begin{split}i=0\ \ \  \\ i\in 4\mathbb{Z}+3 \end{split}}}^{k}W(k,i)^{-}\otimes V_{i+1,r}.\end{eqnarray*}
 for $1\leq r\leq k+1$.

 \begin{eqnarray*}
   \overline{U^{odd,k,2}_{r}}=\bigoplus_{\tiny{\begin{split} i=0 \ \ \ \ \\ i\in 4\mathbb{Z}+1 \ \ \\ i\neq \frac{k}{2} \ \ \end{split}}}^{k}W(k,i)^{-}\otimes V_{i+1,r}\oplus\big(W(k,\frac{k}{2})^{-}\oplus \widetilde{W(k,\frac{k}{2}}^{+})\big)\otimes V_{\frac{k}{2}+1,r}\oplus\bigoplus_{\tiny{\begin{split}i=0\ \ \  \\ i\in 4\mathbb{Z}+3 \end{split}}}^{k}W(k,i)^{+}\otimes V_{i+1,r}.\end{eqnarray*}
 for $1\leq r\leq k+1$.

  (ii) If $\frac{k}{2}\in 4\mathbb{Z}+2$,

 \begin{eqnarray*}
 \overline{U^{odd,k,1}_{r}}=\bigoplus_{\tiny{\begin{split} i=0 \ \ \ \ \ \\ i\in 4\mathbb{Z}+1 \ \ \  \end{split}}}^{k}W(k,i)^{+}\otimes V_{i+1,r}\oplus\big(W(k,\frac{k}{2})^{+}\oplus \widetilde{W(k,\frac{k}{2}}^{-})\big)\otimes V_{\frac{k}{2}+1,r}
 \oplus\bigoplus_{\tiny{\begin{split}i=0\ \ \  \\ i\in 4\mathbb{Z}+3 \\ i\neq \frac{k}{2} \ \ \end{split}}}^{k}W(k,i)^{-}\otimes V_{i+1,r}.\end{eqnarray*}
 for $1\leq r\leq k+1$.

 \begin{eqnarray*}
   \overline{U^{odd,k,2}_{r}}=\bigoplus_{\tiny{\begin{split} i=0 \ \ \ \ \ \\ i\in 4\mathbb{Z}+1 \ \ \end{split}}}^{k}W(k,i)^{-}\otimes V_{i+1,r}\oplus\big(W(k,\frac{k}{2})^{-}\oplus \widetilde{W(k,\frac{k}{2}}^{+})\big)\otimes V_{\frac{k}{2}+1,r}\oplus\bigoplus_{\tiny{\begin{split}i=0\ \ \  \\ i\in 4\mathbb{Z}+3 \\ i\neq \frac{k}{2} \ \ \end{split}}}^{k}W(k,i)^{+}\otimes V_{i+1,r}.\end{eqnarray*}
 for $1\leq r\leq k+1$.

\end{thm}
\begin{proof} From Theorem \ref{thm:affinetwist'}, we know that $M_{r}^{even,+}, M_{r}^{even,-}, M_{r}^{odd,+}, M_{r}^{odd,-}$ and $\overline{M_{r}^{even,+}}$, $\overline{M_{r}^{even,-}}$, $\overline{M_{r}^{odd,+}}, \overline{M_{r}^{odd,-}}$ for $1\leq r\leq 2k$ form a complete list of inequivalent irreducible modules of $L_{\hat{g}}(k,0)^{\sigma}$. Since Com$(V_{\mathbb{Z}\gamma},L_{\hat{sl_2}}(k,0))=K(sl_2,k)$ and Com$(K(sl_2,k),L_{\hat{sl_2}}(k,0))=V_{\mathbb{Z}\gamma}$ \cite{DLY2}, we deduce that
Com$(V_{\mathbb{Z}\gamma}^{\sigma},L_{\hat{sl_2}}(k,0)^{\sigma})=K(sl_2,k)^{\sigma}$ and Com$(K(sl_2,k)^{\sigma},L_{\hat{sl_2}}(k,0)^{\sigma})=V_{\mathbb{Z}\gamma}^{\sigma}$, i.e., Com$(V_{\mathbb{Z}\gamma}^{+},L(k,0)^{+})=K^{+}$ and Com$(K^{+},L(k,0)^{+})=V_{\mathbb{Z}\gamma}^{+}$. And note that
\begin{eqnarray*} \mbox{Com}(V_{\mathbb{Z}\gamma},L_{\hat{\g}}(k,0))=K(\g,k),\ \  \mbox{Com} (K(\g,k),L_{\hat{\g}}(k,0))=V_{\mathbb{Z}\gamma}.\end{eqnarray*}
Thus we have \begin{eqnarray*} \mbox{Com}(V_{\mathbb{Z}\gamma}^{+},L_{\hat{\g}}(k,0)^{\sigma})=K(\g,k)^{\sigma},\ \  \mbox{Com} (K(\g,k)^{\sigma},L_{\hat{\g}}(k,0)^{\sigma})=V_{\mathbb{Z}\gamma}^{+}.\end{eqnarray*}
For (1), if $k=2n+1$, $n\geq 1$, from (\ref{mod1.}), (\ref{affineorbifold1}) and (\ref{affineorbifold2}), we have

\begin{eqnarray*}
\begin{split}M_{r}^{even,+}&=\bigoplus_{i=0, \ i\in 4\mathbb{Z}}^{k}L(k,i)^{+}\otimes V_{i+1,r}\bigoplus_{i=0, \ i\in 4\mathbb{Z}+2}^{k}L(k,i)^{-}\otimes V_{i+1,r}\\
&=\bigoplus_{i=0, \ i\in 4\mathbb{Z}}^{k} \big(\oplus_{0\leq j\leq \frac{i}{2}-1}V_{\mathbb{Z}\gamma+\frac{i-2j}{2k}\gamma}\otimes M^{i,j}\oplus V_{\mathbb{Z}\gamma}^{+}\otimes (M^{i,\frac{i}{2}})^{+}\oplus V_{\mathbb{Z}\gamma}^{-}\otimes (M^{i,\frac{i}{2}})^{-}\\
&\oplus(\oplus_{1\leq j\leq\frac{k-i-1}{2}}V_{\mathbb{Z}\gamma+\frac{-i-2j}{2k}\gamma}\otimes M^{i,i+j})\big)\otimes V_{i+1,r} \bigoplus_{i=0, \ i\in 4\mathbb{Z}+2}^{k}
\big(\oplus_{0\leq j\leq \frac{i}{2}-1}V_{\mathbb{Z}\gamma+\frac{i-2j}{2k}\gamma}\otimes M^{i,j}\\
&\oplus V_{\mathbb{Z}\gamma}^{+}\otimes (M^{i,\frac{i}{2}})^{-}\oplus V_{\mathbb{Z}\gamma}^{-}\otimes (M^{i,\frac{i}{2}})^{+}\oplus(\oplus_{1\leq j\leq\frac{k-i-1}{2}}V_{\mathbb{Z}\gamma+\frac{-i-2j}{2k}\gamma}\otimes M^{i,i+j})\big)\otimes V_{i+1,r},
\end{split}
\end{eqnarray*}
for $1\leq r\leq 2k+2$, where in the second equality, we use the decomposition of $L(k,i)^{\pm}$ as $V_{\mathbb{Z}\gamma}^{+}\otimes K^{+}$-modules. We see that if $k=2n+1$, there are $n$ different $s=i-2j$ for $2\leq i\leq k, i\in 2\mathbb{Z}, 0\leq j\leq\frac{i}{2}-1$. Note that $V_{\mathbb{Z}\gamma}+\frac{i}{2k}\gamma\cong V_{\mathbb{Z}\gamma}-\frac{i}{2k}\gamma$ for $1\leq i\leq k-1$ as $V_{\mathbb{Z}\gamma}^{+}$-modules, thus if we fix $s$, the multiplicity space of the irreducible $V_{\mathbb{Z}\gamma}^{+}$-module $V_{\mathbb{Z}\gamma+\frac{i-2j}{2k}\gamma}$ is $U_{r}^{even,k,s}$ for $1\leq r\leq 2k+2$. The multiplicity space of $V_{\mathbb{Z}\gamma}^{+}$ is $S^{even,k,1}_{r}$ for $1\leq r\leq 2k+2$. The multiplicity space of $V_{\mathbb{Z}\gamma}^{-}$ is $S^{even,k,2}_{r}$ for $1\leq r\leq 2k+2$.  We can obtain the same multiplicity space, i.e., irreducible modules of $K(\g,k)^{\sigma}$ from the decomposition of $M_{r}^{even,-}$. Similarly, we can obtain (\ref{untwistedmod3.})-(\ref{untwistedmod4'.}) from the decomposition of $M_{r}^{odd,\pm}$. From \cite{CKLR}, \cite{CKL}, \cite{CKM}, we know that $U_{r}^{even,k,s}$, $S^{even,k,1}_{r}$, $S^{even,k,2}_{r}$, $U_{r}^{odd,k,s}$, $S^{odd,k,1}_{r}$, $S^{odd,k,2}_{r}$ for $1\leq r\leq 2k+2$ are irreducible modules of $K(\g,k)^{\sigma}$.

For (\ref{twistedmod5.}) and (\ref{twistedmod5'.}), from (\ref{mod3.}), (\ref{twistedmod1}), (\ref{twistedmod2}), we have

\begin{eqnarray*}
\begin{split}\overline{M_{r}^{even,+}}&=\bigoplus_{i=0, \ i\in 4\mathbb{Z}}^{k}\overline{L(k,i)}^{+}\otimes V_{i+1,r}\bigoplus_{i=0, \ i\in 4\mathbb{Z}+2}^{k}\overline{L(k,i)}^{-}\otimes V_{i+1,r}\\
&=\bigoplus_{i=0, \ i\in 4\mathbb{Z}}^{k}\big(V_{\mathbb{Z}\gamma}^{T_{1},+}\otimes W(k,i)^{+}\oplus V_{\mathbb{Z}\gamma}^{T_{1},-}\otimes W(k,i)^{-}\big)\otimes V_{i+1,r}\\
& \bigoplus_{i=0, \ i\in 4\mathbb{Z}+2}^{k}\big(V_{\mathbb{Z}\gamma}^{T_{1},-}\otimes W(k,i)^{+}\oplus V_{\mathbb{Z}\gamma}^{T_{1},+}\otimes W(k,i)^{-}\big)\otimes V_{i+1,r},
\end{split}
\end{eqnarray*}
for $1\leq r\leq 2k+2$, where in the second equality, we use the decomposition of $\overline{L(k,i)^{\pm}}$ as $V_{\mathbb{Z}\gamma}^{+}\otimes K^{+}$-modules. The multiplicity spaces of $V_{\mathbb{Z}\gamma}^{T_{1},+}$ are $\overline{U^{even,k,1}_{r}}$ for $1\leq r\leq 2k+2$. The multiplicity spaces of $V_{\mathbb{Z}\gamma}^{T_{1},-}$ are $\overline{U^{even,k,2}_{r}}$ for $1\leq r\leq 2k+2$. We can obtain the same multiplicity space, i.e., modules of $K(\g,k)^{\sigma}$ from the decomposition of $\overline{M_{r}^{even,-}}$. Similarly, we can obtain (\ref{twistedmod6.}) and (\ref{twistedmod6'.}) from the decomposition of $\overline{M_{r}^{odd,\pm}}$. From \cite{CKLR}, \cite{CKL}, \cite{CKM}, we know that $\overline{U^{even,k,1}_{r}}$, $\overline{U^{even,k,2}_{r}}$ for $1\leq r\leq 2k$ are irreducible modules of $K(\g,k)^{\sigma}$. From \cite{KM}, we know that (\ref{untwistedmod1.})-(\ref{twistedmod6'.}) exhaust all the inequivariant irreducible modules of $K(\g,k)^{\sigma}$.

For (2), if $k=4n$, $n\geq 1$, from (\ref{mod1.}) and Lemma \ref{decomp.}(4), we have

\begin{eqnarray*}
\begin{split}M_{r}^{even,+}&=\bigoplus_{i=0, \ i\in 4\mathbb{Z}}^{k}L(k,i)^{+}\otimes V_{i+1,r}\bigoplus_{i=0, \ i\in 4\mathbb{Z}+2}^{k}L(k,i)^{-}\otimes V_{i+1,r}\\
&=\bigoplus_{i=0, \ i\in 4\mathbb{Z}}^{k} \big(\oplus_{0\leq j\leq \frac{i}{2}-1}V_{\mathbb{Z}\gamma+\frac{i-2j}{2k}\gamma}\otimes M^{i,j}\oplus V_{\mathbb{Z}\gamma}^{+}\otimes (M^{i,\frac{i}{2}})^{+}\oplus V_{\mathbb{Z}\gamma}^{-}\otimes (M^{i,\frac{i}{2}})^{-}\\
&\oplus(\oplus_{1\leq j\leq\frac{k-i}{2}-1}V_{\mathbb{Z}\gamma+\frac{-i-2j}{2k}\gamma}\otimes M^{i,i+j})\oplus V_{\mathbb{Z}\gamma+\frac{1}{2}\gamma}^{+}\otimes (M^{k-i,\frac{k-i}{2}})^{+}\oplus V_{\mathbb{Z}\gamma+\frac{1}{2}\gamma}^{-}\otimes (M^{k-i,\frac{k-i}{2}})^{-}\big)\\
&\otimes V_{i+1,r}\bigoplus_{i=0, \ i\in 4\mathbb{Z}+2}^{k}
\big(\oplus_{0\leq j\leq \frac{i}{2}-1}V_{\mathbb{Z}\gamma+\frac{i-2j}{2k}\gamma}\otimes M^{i,j}\oplus V_{\mathbb{Z}\gamma}^{+}\otimes (M^{i,\frac{i}{2}})^{-}\oplus V_{\mathbb{Z}\gamma}^{-}\otimes (M^{i,\frac{i}{2}})^{+}\\
&\oplus(\oplus_{1\leq j\leq\frac{k-i}{2}-1}V_{\mathbb{Z}\gamma+\frac{-i-2j}{2k}\gamma}\otimes M^{i,i+j})\oplus V_{\mathbb{Z}\gamma+\frac{1}{2}\gamma}^{+}\otimes (M^{k-i,\frac{k-i}{2}})^{-}\oplus V_{\mathbb{Z}\gamma+\frac{1}{2}\gamma}^{-}\otimes (M^{k-i,\frac{k-i}{2}})^{+}\big)\otimes V_{i+1,r},
\end{split}
\end{eqnarray*}
for $1\leq r\leq 2k+2$, where in the second equality, we use the decomposition of $L(k,i)^{\pm}$ as $V_{\mathbb{Z}\gamma}^{+}\otimes K^{+}$-modules. We see that if $k=4n$, there are $(2n-1)$ different $s=i-2j$ for $2\leq i\leq k, i\in 2\mathbb{Z}, 0\leq j\leq\frac{i}{2}-1, (i,j)\neq (k,0)$. Note that $V_{\mathbb{Z}\gamma}+\frac{i}{2k}\gamma\cong V_{\mathbb{Z}\gamma}-\frac{i}{2k}\gamma$ for $1\leq i\leq k-1$ as $V_{\mathbb{Z}\gamma}^{+}$-modules, thus if we fix $s$, the multiplicity space of the irreducible $V_{\mathbb{Z}\gamma}^{+}$-module $V_{\mathbb{Z}\gamma+\frac{i-2j}{2k}\gamma}$ is $U_{r}^{even,k,s}$ for $1\leq r\leq 2k+2$. The multiplicity space of $V_{\mathbb{Z}\gamma}^{+}$ is $S^{even,k,1}_{r}$ for $1\leq r\leq 2k+2$. The multiplicity space of $V_{\mathbb{Z}\gamma}^{-}$ is $S^{even,k,2}_{r}$ for $1\leq r\leq 2k+2$. The multiplicity space of $V_{\mathbb{Z}\gamma+\frac{1}{2}\gamma}^{+}$ is
\begin{eqnarray*}
S^{even,k,1'}_{r}=\bigoplus_{\tiny{\begin{split} i=0 \ \ \ \\ i\in 4\mathbb{Z} \ \ \end{split}}}^{k}(M^{k-i,\frac{k-i}{2}})^{+}\otimes V_{i+1,r}\oplus\bigoplus_{\tiny{\begin{split}i=0\ \ \  \\ i\in 4\mathbb{Z}+2 \end{split}}}^{k}(M^{k-i,\frac{k-i}{2}})^{-}\otimes V_{i+1,r}.\end{eqnarray*}
  for $1\leq r\leq 2k+2$.

 The multiplicity space of $V_{\mathbb{Z}\gamma+\frac{1}{2}\gamma}^{-}$ is  \begin{eqnarray*}S^{even,k,2'}_{r}=
\bigoplus_{\tiny{\begin{split} i=0 \ \ \ \\ i\in 4\mathbb{Z} \ \ \end{split}}}^{k}(M^{k-i,\frac{k-i}{2}})^{-}\otimes V_{i+1,r}\oplus\bigoplus_{\tiny{\begin{split}i=0\ \ \  \\ i\in 4\mathbb{Z}+2 \end{split}}}^{k}(M^{k-i,\frac{k-i}{2}})^{+}\otimes V_{i+1,r}.\end{eqnarray*}
 for $1\leq r\leq 2k+2$.
 Since $k=4n$, we see that $S^{even,k,1}_{r}=S^{even,k,1'}_{r}$ and $S^{even,k,2}_{r}=S^{even,k,2'}_{r}$ for $1\leq r\leq 2k+2$. We can obtain the same multiplicity space, i.e., irreducible modules of $K(\g,k)^{\sigma}$ from the decomposition of $M_{r}^{even,-}$. Similarly, we can obtain (\ref{untwistedmod7'.}) from the decomposition of $M_{r}^{odd,\pm}$. From \cite{CKLR}, \cite{CKL}, \cite{CKM}, we know that $U_{r}^{even,k,s}$, $U_{r}^{odd,k,s}$, $S^{even,k,1}_{r}$, $S^{even,k,2}_{r}$ for $1\leq r\leq 2k+2$ are irreducible modules of $K(\g,k)^{\sigma}$.

 For (\ref{untwistedmod10.}), Since $\overline{U^{odd,k}_{r}}$ are the multiplicity spaces of $V_{\mathbb{Z}\gamma}^{T_{2},+}$ from the decomposition of $M_{r}^{odd,+}$ for $1\leq r\leq 2k+2$. In this case, we note that the multiplicity spaces of $V_{\mathbb{Z}\gamma}^{T_{2},-}$ from the decomposition of $M_{r}^{odd,+}$ are \begin{eqnarray*}
 \overline{U^{odd,k,'}_{r}}=\bigoplus_{\tiny{\begin{split} i=0 \ \ \ \\ i\in 4\mathbb{Z}+1  \ \end{split}}}^{k}W(k,i)^{-}\otimes V_{i+1,r}\oplus\bigoplus_{\tiny{\begin{split}i=0\
  \ \  \\ i\in 4\mathbb{Z}+3 \end{split}}}^{k}W(k,i)^{+}\otimes V_{i+1,r}\end{eqnarray*}
  for $1\leq r\leq 2k+2$. Since $W(k,i)^{+}=W(k,k-i)^{+}$, $W(k,i)^{-}=W(k,k-i)^{-}$ as $K^{+}$-modules for $0\leq i\leq k$ and $V_{i+1,r}\cong V_{k-i+1,2k+3-r}$ as $L^{Vir}(c_{2k+3,k+2},0)$-modules, where $L^{Vir}(c_{2k+3,k+2},0)$ is the minimal Virasoro vertex operator algebra with central charge $c_{2k+3,k+2}=1-\frac{6(k+1)^{2}}{(2k+3)(k+2)}$. It implies that $\overline{U^{odd,k}_{r}}\cong\overline{U^{odd,k,'}_{r}}$ for $1\leq r \leq 2k+2$.

  Note that $k=4n$, (i) If $\frac{k}{2}\in 4\mathbb{Z}$, from (\ref{mod3.}) and Lemma \ref{twisted} (4), we have

\begin{eqnarray*}
\begin{split}\overline{M_{r}^{even,+}}&=\bigoplus_{i=0, \ i\in 4\mathbb{Z},\ i\neq \frac{k}{2}}^{k}\overline{L(k,i)}^{+}\otimes V_{i+1,r}\oplus\overline{L(k,\frac{k}{2})}^{+}\otimes V_{\frac{k}{2}+1,r}\oplus\bigoplus_{i=0, \ i\in 4\mathbb{Z}+2}^{k}\overline{L(k,i)}^{-}\otimes V_{i+1,r}\\
&=\bigoplus_{i=0, \ i\in 4\mathbb{Z}}^{k}\big(V_{\mathbb{Z}\gamma}^{T_{1},+}\otimes W(k,i)^{+}\oplus V_{\mathbb{Z}\gamma}^{T_{1},-}\otimes W(k,i)^{-}\big)\otimes V_{i+1,r}\\
&\oplus\big(V_{\mathbb{Z}\gamma}^{T_{1},+}\otimes (W(k,\frac{k}{2})^{+}\oplus\widetilde{W(k,\frac{k}{2})}^{-})\oplus V_{\mathbb{Z}\gamma}^{T_{1},-}\otimes (W(k,\frac{k}{2})^{-}\oplus\widetilde{W(k,\frac{k}{2})}^{+})\big)\otimes V_{\frac{k}{2}+1,r}\\
& \bigoplus_{i=0, \ i\in 4\mathbb{Z}+2}^{k}\big(V_{\mathbb{Z}\gamma}^{T_{1},-}\otimes W(k,i)^{+}\oplus V_{\mathbb{Z}\gamma}^{T_{1},+}\otimes W(k,i)^{-}\big)\otimes V_{i+1,r},
\end{split}
\end{eqnarray*}
for $1\leq r\leq 2k+2$, where in the second equality, we use the decomposition of $\overline{L(k,i)^{\pm}}$ as $V_{\mathbb{Z}\gamma}^{+}\otimes K^{+}$-modules. The multiplicity spaces of $V_{\mathbb{Z}\gamma}^{T_{1},+}$ are $\overline{U^{even,k,1}_{r}}$ for $1\leq r\leq 2k+2$. Since $W(k,i)^{+}=W(k,k-i)^{+}$, $W(k,i)^{-}=W(k,k-i)^{-}$ as $K^{+}$-modules for $0\leq i\leq k$ and $V_{i+1,r}\cong V_{k-i+1,2k+3-r}$ as $L^{Vir}(c_{2k+3,k+2},0)$-modules, it shows that $\overline{U^{even,k,1}_{r}}$ for $1\leq r\leq k$ are inequivalent modules among $\overline{U^{even,k,1}_{r}}$ for $1\leq r\leq 2k+2$. Similarly, the multiplicity spaces of $V_{\mathbb{Z}\gamma}^{T_{1},-}$ are $\overline{U^{even,k,2}_{r}}$ for $1\leq r\leq 2k+2$ and $\overline{U^{even,k,2}_{r}}$ for $1\leq r\leq k$ are inequivalent modules among $\overline{U^{even,k,2}_{r}}$ for $1\leq r\leq 2k+2$. We can obtain the same multiplicity space, i.e., modules of $K(\g,k)^{\sigma}$ from the decomposition of $\overline{M_{r}^{even,-}}$.  From \cite{CKLR}, \cite{CKL}, \cite{CKM}, we know that $\overline{U^{even,k,1}_{r}}$, $\overline{U^{even,k,2}_{r}}$ for $1\leq r\leq k$ are inequivalent irreducible modules of $K(\g,k)^{\sigma}$. The case (ii) for $\frac{k}{2}\in 4\mathbb{Z}+2$ is similar to prove as the case (i). And the case (3) for $k=4n+2$, $n\geq 1$  is similar to prove as the case (2).

\end{proof}

From the Theorem 5.4 of \cite{JW2} and the irreducible modules of $L^{Vir}(c_{2k+3,k+2},0)$ are self-dual. We obtain:

 \begin{prop} All irreducible modules of the orbifold vertex operator algebra $K(\g,k)^{\sigma}$ are self-dual.

 \end{prop}

\end{document}